\documentclass[12pt]{article}
\usepackage{amsmath, amsthm, amssymb,latexsym,enumerate}
\usepackage{graphicx}
\usepackage{hyperref}
\usepackage{xcolor}
\usepackage{cite}
\usepackage{dsfont}
\usepackage{graphicx, caption, subcaption}
\usepackage{soul}

\newtheorem{theorem}{Theorem}[section]
\newtheorem{lemma}[theorem]{Lemma}
\newtheorem{corollary}[theorem]{Corollary}
\newtheorem{claim}[theorem]{Claim}
\newtheorem{subclaim}[theorem]{Subclaim}
\newtheorem{conjecture}[theorem]{Conjecture}

\newtheorem*{thmm:chordal}{Theorem~\ref{thm:chordal}}
\newtheorem*{thmm:lcc}{Theorem~\ref{thm:lcc}}
\newtheorem*{corr:ql}{Theorem~\ref{cor:ql}}
\newtheorem*{thmm:claw}{Theorem~\ref{thm:claw}}
\newtheorem*{thmm:genclaw}{Theorem~\ref{thm:genclaw}}

\numberwithin{equation}{section}

\newcommand{\al}{\alpha}
\newcommand{\be}{\beta}
\newcommand{\vp}{\varphi}

\newcommand{\uc}{\vp_*}
\newcommand{\odd}{\chi_{\mathrm{o}}}
\newcommand{\pcf}{\chi_{\mathrm{pcf}}}
\newcommand{\hpcf}{\chi^h_{\mathrm{pcf}}}
\newcommand{\D}{\Delta}

\DeclareMathOperator{\lcc}{lcc}
\DeclareMathOperator{\RN}{RN}
\DeclareMathOperator{\FC}{FC}

\title{Brooks-type theorems for relaxations of square colorings}

\author{Eun-Kyung Cho\thanks{
Department of Mathematics, Hankuk University of Foreign Studies, Yongin-si, Gyeonggi-do, Republic of Korea.
 \texttt{ekcho2020@gmail.com}
}
\and Ilkyoo Choi\thanks{
Department of Mathematics, Hankuk University of Foreign Studies, Yongin-si, Gyeonggi-do, Republic of Korea.
\texttt{ilkyoo@hufs.ac.kr}
}
\and Hyemin Kwon\thanks{
Department of Mathematics, Ajou University, Suwon-si, Gyeonggi-do, Republic of Korea.
\texttt{khmin1121@ajou.ac.kr}
}
\and Boram Park\thanks{
Department of Mathematics, Ajou University, Suwon-si, Gyeonggi-do, Republic of Korea.
\texttt{borampark@ajou.ac.kr}
}}
\date\today

\begin{document}
 
\maketitle
\begin{abstract}
The following relaxation of proper coloring the square of a graph was recently introduced: 
for a positive integer $h$, the {\it proper $h$-conflict-free chromatic number} of a graph $G$, denoted $\hpcf(G)$, is the minimum $k$ such that $G$ has a proper $k$-coloring where every vertex $v$ has $\min\{\deg_G(v),h\}$ colors appearing exactly once on its neighborhood.
Caro, Petru\v{s}evski, and \v{S}krekovski put forth a Brooks-type conjecture: if $G$ is a graph with $\D(G)\ge 3$, then $\pcf^1(G)\leq \D(G)+1$.
The best known result regarding the conjecture is $\pcf^1(G)\leq 2\D(G)+1$, which is implied by a result of Pach and Tardos. 
We improve upon the aforementioned result for all $h$, and also enlarge the class of graphs for which the conjecture is known to be true.

Our main result is the following: for a graph $G$, if $\D(G) \ge h+2$, then $\hpcf(G)\le (h+1)\D(G)-1$; this is tight up to the additive term as we  explicitly construct infinitely many graphs $G$ with $\hpcf(G)=(h+1)(\D(G)-1)$.
We also show that 
the conjecture is true for chordal graphs, and obtain partial results for quasi-line graphs and claw-free graphs.
Our main result also improves upon a Brooks-type result for $h$-dynamic coloring. 
\end{abstract}

\section{Introduction and main results} 
All graphs in this paper are finite and simple.
For a graph $G$, let $V(G)$ and $E(G)$ denote its vertex set and edge set, respectively.
For a vertex $v$, let $\deg_G(v)$ and $N_G(v)$ denote the degree and the neighborhood, respectively, of $v$. 
The maximum degree of (a vertex of) a graph $G$ is denoted by $\D(G)$. 
For a positive integer $k$, a {\it proper $k$-coloring} of a graph $G$ is a function on $V(G)$ that assigns each vertex a color in $\{1, \ldots, k\}$ so that the end points of every edge receive different colors. 
The minimum $k$ for which $G$ has a proper $k$-coloring is the {\it chromatic number} of $G$, denoted $\chi(G)$. 
A greedy coloring algorithm according to an arbitrary ordering of the vertices of a graph $G$ guarantees that $G$ has a proper $(\D(G)+1)$-coloring, so $\chi(G)\leq \D(G)+1$.
Since every edge acts as a restriction, it is natural to seek an upper bound on the chromatic number in terms of the maximum degree. 
These types of results are known as Brooks-type theorems, as 
Brooks~\cite{1941Brooks} characterized all cases where equality holds. 

Given a graph $G$, the {\it square} of $G$, denoted $G^2$, is obtained by adding all edges between vertices of distance at most 2 in $G$. 
Coloring the square of a graph is a popular avenue of research, see~\cite{2008KrKr} for a survey on graph coloring with distance constraints and~\cite{1995JeTo} for a prominent book on graph coloring. 
For a graph $G$, a vertex of maximum degree and all its neighbors must receive distinct colors in a proper coloring of $G^2$, so a trivial lower bound on $\chi(G^2)$ is $\D(G)+1$.
Thus, we obtain the following chain of inequalities: $\chi(G)\leq\D(G)+1\leq\chi(G^2)\leq\D(G^2)+1\leq (\D(G) )^2+1$. 
The last inequality holds by simply counting all possible neighbors of a maximum degree vertex. 
Recently, two new coloring parameters that are stronger than a proper coloring of a graph $G$, but weaker than a proper coloring of $G^2$, were introduced and gained immediate attention. 
We are interested in obtaining Brooks-type results regarding these new coloring parameters.

The first is the concept of an  odd coloring, formalized by Petru\v{s}evski and \v{S}krekovski~\cite{petrusevski2021colorings} as a strengthening of proper coloring. 
For a positive integer $k$, an {\it odd $k$-coloring} of a graph $G$ is a proper $k$-coloring of $G$ such that every vertex of positive degree has a color appearing an odd number of times on its neighborhood. 
The minimum $k$ for which $G$ has an odd $k$-coloring is the {\it odd chromatic number} of $G$, denoted $\odd(G)$. 
If $\chi_o(G)\le k$, then $G$ is \textit{odd $k$-colorable}. 
Ever since the first paper by Petru\v{s}evski and \v{S}krekovski~\cite{petrusevski2021colorings} appeared, there have been numerous papers~\cite{caro2022remarksodd,2023ChChKwBo,cranston2022note,dujmovic2022odd,liu20221,metrebian2022odd,petr2022odd,qi2022odd,tian2022every,tian2022every2,tian2022odd,wang2022odd,cranston2022odd} studying various aspects of this new coloring concept across several graph classes. 
In particular, Caro, Petru\v{s}evski, and \v{S}krekovski~\cite{caro2022remarksodd} conjectured the following Brooks-type statement; namely, the same number of colors used by the greedy coloring algorithm for proper coloring is sufficient for odd coloring a graph.

\begin{conjecture}[\cite{caro2022remarksodd}]\label{conj:odd-brooks}
    If $G$ is a graph with $\D(G)\geq 3$, then $\odd(G)\leq \D(G)+1$. 
\end{conjecture}
Note that the lower bound on the maximum degree is necessary since a $5$-cycle $C_5$ has maximum degree $2$, yet $\odd(C_5)=5$. 
In~\cite{caro2022remarksodd}, the authors prove  Conjecture~\ref{conj:odd-brooks} for graphs with maximum degree 3.
They also show that every graph $G$ has an odd $2\D(G)$-coloring, unless $G$ is a $5$-cycle, and the bound of $2\D(G)$ is tight for infinitely many cycles. 

The second came shortly after, and is a strengthening of odd coloring  named proper conflict-free coloring, introduced by Fabrici, Lu\v{z}ar, Rindo\v{s}ov\'{a}, and Sot\'{a}k~\cite{FABRICI202380}. 
For a positive integer $k$, a {\it proper conflict-free $k$-coloring} of a graph $G$ is a proper $k$-coloring of $G$ such that every vertex of positive degree has a color appearing exactly once on its neighborhood. 
The minimum $k$ for which $G$ has a proper conflict-free $k$-coloring is the {\it proper conflict-free chromatic number} of $G$, denoted $\pcf(G)$. 
This new coloring parameter also received considerable attention~\cite{arXiv_Liu,CARO2023113221,arXiv_Hickingbotham,ahn2022proper,cho2022proper,cranston2022proper,kamyczura2022conflict,cho2022relaxation}.
The originators of Conjecture~\ref{conj:odd-brooks}  put forth the following additional Brooks-type conjecture regarding the proper conflict-free chromatic number.
Namely, they conjectured $\D(G)+1$ is between $\pcf(G)$ and $\chi(G^2)$ for a graph $G$ when $\D(G)\geq 3$.  

\begin{conjecture}[\cite{CARO2023113221}]\label{conj:pcf-brooks}
    If $G$ is a graph with $\D(G)\geq 3$, then $\pcf(G)\leq \D(G)+1$. 
\end{conjecture}

Note that by definition, $\chi(G)\leq \odd(G)\leq \pcf(G)\leq \chi(G^2)$ holds for every graph $G$.
Thus, Conjecture~\ref{conj:pcf-brooks} is a stronger statement than Conjecture~\ref{conj:odd-brooks}. 
In~\cite{CARO2023113221}, the authors pointed out that Conjecture~\ref{conj:pcf-brooks} is true for graphs with maximum degree 3 by a result on superlinear colorings~\cite{liu2013linear}.  
The authors themselves proved that every graph $G$ has a proper conflict-free $\left\lfloor 2.5\D(G)\right\rfloor$-coloring and equality holds if and only if $G\in\{K_2, C_5\}$. 
They also showed that every graph $G$ that is either chordal or claw-free has a proper conflict-free $(2\D(G)+1)$-coloring. 
Recall that a $5$-cycle $C_5$ requires $2\D(C_5)+1$ colors for an odd coloring, so also for a proper conflict-free coloring. 

Actually, a result by Pach and Tardos~\cite{2009Pach} already implies that every graph $G$ has a proper conflict-free $(2\D(G)+1)$-coloring:
given a graph $G$, define the hypergraph $H_G$ so that 
$V(H_G)=V(G)$ and $E(H_G)=E(G)\cup\{N_G(v)\mid v\in V(G)\}$.
Note that $\Delta(H_G)\le 2\Delta(G)$, and a conflict-free coloring of $H_G$ naturally induces a proper conflict-free coloring of $G$. 
Since every hypergraph $H$ has a conflict-free coloring with $\D(H)+1$ colors by~\cite{2009Pach}, 
we obtain $\pcf(G) \le 2\Delta(G)+1$. 
This was also pointed out by Cranston and Liu in~\cite{cranston2022proper}, where they proved that a graph $G$ with sufficiently large maximum degree (at least $10^{9}$) satisfies $\pcf(G)\leq \left\lceil1.6550826\D(G)+\sqrt{\D(G)}\right\rceil$, even for the more general list version. 
Recently, Kamyczura and Przyby{\l}o~\cite{kamyczura2022conflict} proved Conjecture~\ref{conj:pcf-brooks} asymptotically for graphs with large enough minimum degree.

As a step towards Conjecture~\ref{conj:pcf-brooks} (and also Conjecture~\ref{conj:odd-brooks}), we improve upon the best known result of $\pcf(G)\leq 2\D(G)+1$ for a graph $G$.
We actually go much further and obtain a result for 
a strengthening of proper conflict-free coloring, which we call “proper $h$-conflict-free coloring" where $h$ is a positive integer. 
For a positive integer $k$, a {\it proper $h$-conflict-free $k$-coloring} (or a proper $h$-CF $k$-coloring for short) of a graph $G$ is a proper $k$-coloring of $G$ such that every vertex $v$ has $h_v:=\min\{\deg_G(v),h\}$ colors appearing exactly once on its neighborhood. 
The {\it proper $h$-conflict-free chromatic number} of $G$, denoted $\hpcf(G)$, is the minimum $k$ such that $G$ has a proper $h$-CF $k$-coloring.  
Note that $\hpcf(G)=\chi(G^2)$ whenever $h\ge \Delta(G)-1$. 
We now state our main result:

\begin{theorem}\label{thm:pcf:h-ver2}
Let $h\ge 1$. 
If $G$ is a graph with $\Delta(G)\geq h+2$, then $\hpcf(G)\le (h+1)\Delta(G)-1$.
\end{theorem}

Theorem~\ref{thm:pcf:h-ver2} is tight up to the additive term, as we can construct a graph $G$ that requires $(h+1)(\D(G)-1)$ colors in a proper $h$-CF coloring. 
For a positive integer $n$, let $[n]=\{1,\ldots,n\}$. Let $n\ge 2$ be a positive integer such that $L(n)=n-1$, where $L(n)$ is the maximum size of a family of orthogonal Latin squares of order $n$.  Let $L_1$, $\ldots,$ $L_{n-1}$ be a family of orthogonal Latin squares of order $n$, and for each $k\in[n-1]$ and $i,j\in[n]$, let $L_k(i,j)$ be the symbol in the $(i,j)$-position of $L_k$.
Define the graph $G_n$ as:
\begin{eqnarray*}
V(G_n) &=& \{ v_{i,j}\mid i,j\in [n]\} \cup \{s_{k,i}\mid k\in[n-1],i\in[n]\}\cup\{r_i\mid i\in[n]\}\cup\{c_i\mid i\in[n]\},\\
E(G_n)&=&\bigcup_{i,j\in [n]} \left( \{ v_{i,j}r_i, v_{i,j}c_j\}\cup \{v_{i,j}s_{k,L_k(i,j)} \mid k\in [n-1]\} \right).
\end{eqnarray*}
See Figure~\ref{Latin-square} for an illustration. 

\begin{figure}[h!]
\centering
\includegraphics[width=15cm]{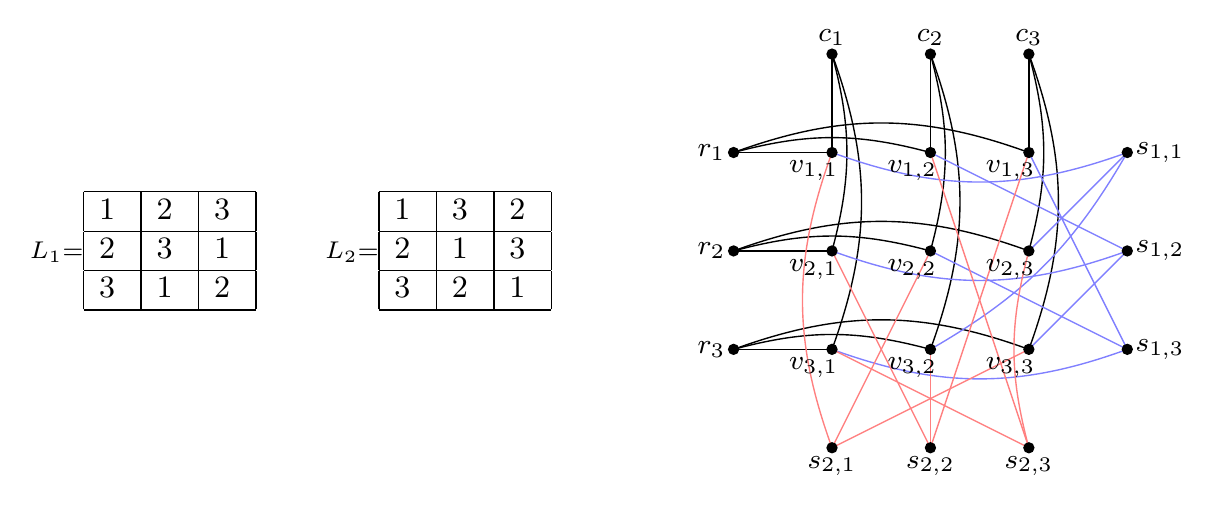}
\caption{Orthogonal Latin squares $L_1$ and $L_2$ of order $3$, and the graph $G_3$}
\label{Latin-square}
\end{figure}

Note that $\deg_{G_n}(v_{i,j})=n+1$, and  $\deg_{G_n}(r_i)=\deg_{G_n}(c_i)=\deg_{G_n}(s_{k,i})=n$, so $\D(G_n)=n+1$.
In a proper $(n-1)$-CF coloring of $G_n$,
all vertices $v_{i,j}$ receive distinct colors, so $\pcf^{n-1}(G_n)\ge n^2 = n(\Delta(G_n)-1)$.
If $n$ is a prime power, then it is well-known that $L(n)=n-1$, so there are infinitely many pairs of a graph $G$ and an integer $h$ for which $\hpcf(G)\ge (h+1)(\Delta(G)-1)$.

We take a slight detour and introduce an application of Theorem~\ref{thm:pcf:h-ver2}. For positive integers $h$ and $k$,
an {\it $h$-dynamic $k$-coloring} of a graph $G$ is a proper $k$-coloring of $G$ such that every vertex $v$ has at least $h_v$ colors appearing on its neighborhood. 
Thus, a proper $h$-CF coloring is an $(h+1)$-dynamic coloring, so the following corollary of Theorem~\ref{thm:pcf:h-ver2} is an improvement of a Brooks-type result by Jahanbekam, Kim, O, and West~\cite{jahanbekam2016r} who proved every graph $G$ has an $h$-dynamic $(h\D(G)+1)$-coloring. 

\begin{corollary}\label{cor:rD}
Every graph $G$ has an $h$-dynamic $(h\D(G)-1)$-coloring for a positive integer $2\le h\le \D(G)-1$.
\end{corollary}

Jendrol' and Onderko~\cite{arXiv_hued} recently showed that every graph $G$ has an $h$-dynamic $ ((h - 1)(\D(G) + 1) + 2)$-coloring.
Corollary~\ref{cor:rD} gives a better upper bound when $h=\D(G)-1$  and the same bound when $h=\D(G)-2$.

Back to partial results regarding Conjectures~\ref{conj:odd-brooks} and~\ref{conj:pcf-brooks}, we obtain the following results on some well-known graph classes.
We first prove a result regarding proper $h$-CF coloring that implies Conjecture~\ref{conj:pcf-brooks} is true for chordal graphs.
Recall that a chordal graph $G$ has a \textit{simplicial ordering} $v_1,\ldots,v_n$ of the vertices of $G$, which means that for each $i\in[n]$, $v_i$ and its neighbors in $G[\{v_i,\ldots, v_n\}]$ form a clique. 
For a vertex $v_i$, the maximal clique containing $v_i$ in the subgraph of $G$ induced by $\{v_i,v_{i+1},\ldots,v_n\}$ is  a \textit{simplicial clique} of $G$. Let $s(G)$ denote the maximum size of all possible simplicial cliques of $G$.
 
\begin{theorem}\label{thm:chordal}
For a positive integer $h$, if $G$ is a chordal graph, then  $\hpcf(G) \le 1 + (h+1) \cdot \min \left\{  s(G)-1,  \frac{\D(G)+h-1}{2}\right\}.$
\end{theorem}

When $h=1$, Theorem~\ref{thm:chordal} implies Conjecture~\ref{conj:pcf-brooks} is true for chordal graphs.

In addition, we study graphs with bounded local vertex clique cover number.  The \textit{local vertex clique cover number}  of a graph $G$, denoted $\lcc(G)$, is the minimum $q$ such that for every vertex $v$, there are at most $q$ cliques whose union is $N_G(v)$.
We obtain the following result:
 
\begin{theorem}\label{thm:lcc}
If $G$ is a graph with $\lcc(G)\le \ell$, then
$\chi_{o}(G)\le \frac{(2\ell-1)}{\ell}\D(G)+2$.
\end{theorem} 

A graph $G$ with $\lcc(G)\le 2$ is  \textit{quasi-line}. 
For a quasi-line graph $G$ with an odd coloring, one color cannot appear on three neighbors of a vertex, so $\odd(G)=\pcf(G)$.
Thus, the following corollary is an immediate consequence of Theorem~\ref{thm:lcc}. 

\begin{corollary}\label{cor:ql}
If $G$ is a quasi-line graph,  then $\pcf(G) =\odd(G)\le \frac{3}{2}\D(G)+2$.  
\end{corollary}

The family of $K_{1, \ell+1}$-free graphs contains graphs with local vertex clique cover number at most $\ell$.

We extend Theorem~\ref{thm:lcc} to $K_{1, \ell+1}$-free graphs and obtain a slightly weaker bound, but the highest order term is identical. 

\begin{theorem}\label{thm:genclaw}
If $G$ is a  $K_{1,\ell+1}$-free graph, 
then $\odd(G) \le \frac{(2\ell-1)}{\ell}\D(G)  +\left\lceil\sqrt{\frac{2\D(G)}{\ell}+1}\right\rceil+1$.
\end{theorem}

A $K_{1,3}$-free graph  is {\it claw-free}. 
Note that the family of claw-free graphs contains quasi-line graphs.
For a claw-free graph, we improve upon the previous theorem by shaving off the additive term of $1$. 
 
\begin{theorem}\label{thm:claw}
If $G$ is a claw-free graph,
then $\pcf(G) = \odd(G) \le \frac{3}{2}\D(G)  +\left\lceil\sqrt{\D(G)}\right\rceil$.
\end{theorem}

Note that Corollary~\ref{cor:ql} and Theorem~\ref{thm:claw} are tight by a $5$-cycle. 

We prove Theorem~\ref{thm:pcf:h-ver2} in Section~\ref{sec:main},
and the proofs of theorems regarding graph classes are presented in Section~\ref{sec:lcc}.
Our proof idea is mostly based on providing a good vertex ordering of a graph, and then showing that an almost greedy coloring algorithm can be utilized to obtain the desired coloring. 
 For a (partial) proper  coloring $\vp$ of a graph $G$, and a vertex $v \in V(G)$, let $\uc(v)$ be a set of $h_v$ unique colors in the neighborhood of $v$ if it exists. 
 We say ``a vertex $v$ has $m$ unique colors under $\vp$'' if there are $m$ unique colors in the neighborhood of $v$ under $\vp$.

 \section{Proof of Theorem~\ref{thm:pcf:h-ver2}}\label{sec:main}

In this section we prove Theorem~\ref{thm:pcf:h-ver2}, whose proof is split into two subsections, subsections~\ref{subsec:largeh} and
~\ref{subsec:smallh}, depending on the maximum degree of the graph. 
Let $G$ be a counterexample to Theorem~\ref{thm:pcf:h-ver2} where $|V(G)|+|E(G)|$ is minimum.
We abuse notation and let $\D$ denote the maximum degree of $G$. 
Let $\mathcal{C}$ be the set of $(h+1)\D-1$ colors for a (partial) proper $((h+1)\D-1)$-coloring  of $G$.

\subsection{The case when $h = \D-2$}\label{subsec:largeh}

We will prove Theorem~\ref{thm:pcf:h-ver2} when $h=\D-2$.
It is known that $\pcf^1(G) \le 4$ when $\D=3$ (see~\cite{CARO2023113221}), so we may assume $\D \geq 4$.  
Note that $\D+(\D-1)(\D-2)\leq \D^2-\D-2<(h+1)\D-1$ since $\D\geq 4$.

Let $\vp$ be a partial proper  $(\D^2-\D-1)$-coloring of $G$.
For a vertex $x$ with an uncolored neighbor, let 
$$F_{\vp}(x)=\begin{cases}
  \uc(x) &\text{if $x$ has $h$ colored neighbors,}\\
   \vp(N_G(x))  &\text{otherwise.}\\
\end{cases}$$
Note that $F_{\vp}(x)\subseteq \vp(N_G(x))$, and  $|F_{\vp}(x)|\le \D-2$. 
For an uncolored vertex $y$ and a subset $S\subseteq N_G(y)$, let 
$B_\vp(y; S)=\vp(N_G(y)) \cup \left( \bigcup_{x\in N_G(y)\setminus S} F_{\vp}(x) \right)$,
so
\begin{eqnarray}\label{eq:h=D-2}
|B_{\vp}(y; S)|&\le& |\vp(N_G(y))|+(\deg_G(y)-|S|)(\D-2).
\end{eqnarray}
Fix a vertex $u$ of degree $\D$.
Let $v$ be a vertex with maximum degree among the neighbors of $u$, and let $H=G-uv$. 
By the minimality of $G$,  $H$ has a proper $h$-CF $(\D^2-\D-1)$-coloring $\vp$.

Suppose that $\deg_G(v)=\D$. 
Note that each of $u$ and $v$ has $\D-1=h+1$ neighbors in $H$ so it has $h+1$ unique colors.
Remove the color on $v$, and continue using $\vp$ as the resulting partial coloring.
Since $|B_\vp(v; \{u\})|\le \D + (\D-1)(\D-2) < \D^2-\D-1$, there is a color $c$ not in $B_\vp(v; \{u\})$.
Now, recolor $v$ with $c$ to obtain a proper $h$-CF $(\D^2-\D-1)$-coloring of $G$, which is a contradiction.

Suppose that $\deg_G(v)\le\D-1$.   
Let $w$ be a neighbor of $u$ other than $v$, so  $\deg_G(w)\le \D-1$.
Now, remove the colors on $u$, $v$, and $w$, and continue using $\vp$ as the resulting coloring.
Then sequentially recolor $u$, $v$, and $w$ with a color in $\mathcal{C}\setminus B_\vp(x; \emptyset)$ for each $x \in \{u, v, w\}$, where each extension is still denoted by $\vp$. 
Note that this is possible since by \eqref{eq:h=D-2}, $|B_\vp(x; \emptyset)| \leq\max\{(\D-2) + \D(\D-2), (\D-1)+(\D-1)(\D-2)\}<\D^2-\D-1$ for each $x\in\{u,v,w\}$.
The resulting coloring is a proper $h$-CF $(\D^2-\D-1)$-coloring of $G$, which is a contradiction.

\subsection{The case when $h\le \D-3$}\label{subsec:smallh}

We will prove Theorem~\ref{thm:pcf:h-ver2} when $h\leq \D-3$.
The main idea is simple: except for the vertices on a minimum length cycle $E_0$, we can color all vertices according to a somewhat natural ordering.
Finishing the coloring by coloring the vertices on $E_0$ is trickier, and requires some in-depth analysis. 

We begin with showing that our minimum counterexample $G$ must be $2$-edge-connected. 
For a (partial) coloring $\vp$ of $G$, let $V(\vp)$ be the set of vertices that are colored under $\vp$.

\begin{lemma}\label{claim:pcf:2edgeconnected}
$G$ is $2$-edge-connected. 
\end{lemma}
\begin{proof} 
Suppose to the contrary that $G$ has a cut-edge $v_1v_2$.
Let $G_i$ be the component of $G-v_1v_2$ containing $v_i$ for $i\in[2]$.
Since $\D(G_i)\le \D$ for each $i\in[2]$, by the minimality of $G$, $G_i$ has a proper $h$-CF $((h+1)\D-1)$-coloring $\vp_i$.
Let $T_{\vp_i}(v_i)$ be a set of $\min\{|N_G(v_i)\cap V(\vp_i)|, h\}$  unique colors in the neighborhood of $v_i$ under $\vp_i$ for each $i\in[2]$.
Now permute the colors on $V(G_2)$ under $\vp_2$ to obtain a coloring $\vp$ of all vertices of $G$ so that $T_{\vp_1}(v_1)\cup \{\vp_1(v_1)\}$ and $T_{\vp_2}(v_2)\cup \{\vp_2(v_2)\}$ are disjoint; this is possible since $|T_{\vp_1}(v_1)|+1+|T_{\vp_2}(v_2)|+1\le 2(h+1) \le (h+1)\D-1$.
Thus, $\vp$ is a proper $h$-CF $((h+1)\D-1)$-coloring of $G$, which is a contradiction. Hence, $G$ is $2$-edge-connected.
\end{proof}

It is well-known that a graph with no cut-edge has an ear decomposition.
We lay out the related definitions for completeness of this article. 
An {\it ear} is a graph that is either a path or a cycle. 
The {\it length} of an ear is the number of edges on it, and a {\it trivial ear} is an ear of length $1$. 
An {\it ear decomposition} of a graph $G$ is a sequence $\mathcal{E}: E_0,E_1,\ldots,E_m$ of ears, where the edge sets of $E_i$'s form a partition of $E(G)$ and $E_0$ is a cycle, which we call the \textit{initial cycle}, and for every $i\in[m]$, $E_i$ is an ear such that $V(E_0\cup E_1\cup \cdots \cup E_{i-1})\cap V(E_i)$ consists of a single vertex and the two endpoints of $E_i$ if $E_i$ is a cycle and a path, respectively. 
The {\it internal vertices} of an ear $E_i$ are the vertices on $E_i$ but not on $E_0\cup E_1\cup\cdots\cup E_{i-1}$.  

\begin{theorem}[\cite{roger1978}]\label{thm:ear}
If a graph is $2$-edge-connected, then an ear decomposition with an arbitrary cycle as the initial ear exists.
\end{theorem}

Using an ear decomposition, we will show that $G$ has a certain ordering of the vertices that will be key in the proof.

\begin{lemma}\label{lemma:pcf:sequence}  
There is a sequence $\sigma:v_1,\ldots,v_n$ of  all vertices of $G$ such that
\begin{itemize}
    \item[(i)] $|N_G(v_i)\cap\{v_{i+1},\ldots,v_n\}|\geq 1$ for each $i\in[n-1]$,
    \item[(ii)] $|N_G(v_i)\cap\{v_{i-1},\ldots,v_n\}|\geq 2$ for each $i\in[n-1]\setminus\{1\}$, and
    \item[(iii)] there exists an $\ell\in [n-2]$ such that $v_{\ell},\ldots, v_n$ form a shortest cycle of $G$ and $v_n$ is a vertex with minimum degree among all vertices on shortest cycles of $G$.
\end{itemize}
\end{lemma}
\begin{proof} 
Since $G$ is $2$-edge-connected by Lemma~\ref{claim:pcf:2edgeconnected}, $G$ has an ear decomposition $\mathcal{E}:E_0,E_1,\ldots,E_m$  where $E_0$ is chosen to be a shortest cycle which contains a vertex with minimum degree among all shortest cycles of $G$; this is possible by Theorem~\ref{thm:ear}. 

For a non-trivial ear $E_i$ where $i\in[m]$,  let $\sigma_i$ be an ordering of the internal vertices of $E_i$ according to a natural ordering on $E_i$, and let $\sigma_0$ be an ordering of the vertices on $E_0$ according to a natural ordering on $E_0$ such that the last vertex of $\sigma_0$ has minimum degree among the vertices of $E_0$.
Now, consider the sequence $\sigma:v_1,\ldots,v_n$ of the vertices of $G$ induced from the ordering $\sigma_m,\sigma_{m-1},\ldots,\sigma_0$ (ignore $\sigma_i$ if it is not defined). 
By the choice of $E_0$, (iii) holds.

For $i\in[n-1]$, let $v_i$ belong to $\sigma_k$ for some $k\in\{0,\ldots,m\}$, and let $xv_iy$ be a subpath of $E_k$ where $x$ comes before $v_i$ or $y$ comes after $v_i$ on $\sigma_k$. 
Note that $x\neq y$. 
If $v_i$ is not the last vertex of $\sigma_k$, then $y$ is also a vertex of $\sigma_k$, so $y=v_{i+1}$.
If $v_i$ is the last vertex of $\sigma_k$, then $y$ belongs to $\sigma_{k'}$ for some $k'<k$, so $y\in\{v_{i+1}, \ldots, v_n\}$.
(Note that $i\neq n$.)
Thus, (i) holds.

For $i\in[n-1]\setminus\{1\}$, if $v_i$ is not the first vertex of $\sigma_k$, then $x$ is also a vertex of $\sigma_k$, so $x=v_{i-1}$.
If $v_i$ is the first vertex of $\sigma_k$, then $x$  belongs to $\sigma_{k'}$ for some $k'<k$, so $x\in \{v_{i+1},\ldots,v_n\}$.
Since $y\in \{v_{i+1},\ldots,v_n\}$ and $x\neq y$, (ii) holds.
\end{proof}

We now define important notation that will be used throughout the proof. 
Let $\vp$ be a partial proper coloring of $G$ using the color set $\mathcal C$.
A vertex $v$ is \textit{$\vp$-good} if either $v$ has $h$ unique colors under $\vp$ or all colored neighbors of $v$ have distinct colors.
We say $\vp$ is \textit{good} if every vertex is $\vp$-good.
For a partial good coloring $\vp$ of $G$ and an uncolored vertex $v$, the set of {\it $\vp$-available colors} of $v$, denoted $A_{\vp}(v)$, is the set of colors $c$ not in $\vp(N_G(v))$ such that the extension of $\vp$ to $v$ with $c$ is also good.  
A vertex $u$ is {\it $\vp$-risky} if $u$ has both an uncolored neighbor and at most $h$ unique colors under $\vp$.
In addition, let $\RN_{\vp}(v)$ be the set of $\vp$-risky neighbors of $v$, and let the set of \textit{$\vp$-forbidden colors} of $v$, denoted $\FC_{\vp}(v)$,  be defined as 
\[ \FC_{\vp}(v): =\vp(N_G(v)) \cup \left( \bigcup_{ u\in\RN_{\vp}(v)}\{\mbox{unique colors in   the neighborhood of $u$ under $\vp$}\} \right).\]
Note that if we can extend $\vp$ to a vertex $v$ with a color that is not a $\vp$-forbidden color of $v$, then $\vp$ is still good. 
By definition,
\begin{eqnarray}\label{eq:A}
A_{\vp}(v)&=& \mathcal{C} \setminus \FC_{\vp}(v).
\end{eqnarray} 
Note that if a vertex $v$ has no $\vp$-available color, then $\FC_{\vp}(v)=\mathcal{C}$.

We now prove important lemmas regarding the above notions that will be frequently used in the proof. 
By \eqref{eq:A}, if $v$ is an uncolored vertex, then  \begin{eqnarray}\label{eq:A:claim}
&&|A_{\vp}(v)|\ge  (\D-|\vp(N_G(v))|) +  h( \D-|\RN_{\vp}(v)|)-1,
\end{eqnarray} 
\noindent
since $(h+1)\D-1-|\vp(N_G(v))| - h|\RN_{\vp}(v)| =(\D-|\vp(N_G(v))|) +  h( \D-|\RN_{\vp}(v)|)-1$.
The following lemma is an easy consequence of \eqref{eq:A:claim}, but we explicitly state it as a lemma as we will refer to it often. 
 
\begin{lemma}\label{lem:uncolored} 
If an uncolored vertex $v$ has no $\vp$-available color, then 
\begin{enumerate}[(i)]
    \item $\deg_G(v)=\D$,
    \item $|\vp(N_G(v))|\ge \D-1$,
    \item $|\RN_{\vp}(v)|\ge \D-1$, and \item $|\vp(N_G(v))|+|\RN_{\vp}(v)|\ge 2\D-1$.
\end{enumerate}
\end{lemma}
\begin{proof}
In all cases, we will obtain that $v$ has a $\vp$-available color, which is a contradiction. 
If $\deg_G(v)\le \D-1$, then $|A_{\vp}(v)|\ge h$ by \eqref{eq:A:claim}.
If $|\vp(N_G(v))|\le \D-2$, then $|A_{\vp}(v)|\ge 1$ by \eqref{eq:A:claim}.
If $|\RN_{\vp}(v)|\le \D-2$, then $|A_{\vp}(v)|\ge 2h-1$ by \eqref{eq:A:claim}. 
If $|\vp(N_G(v))|=|\RN_{\vp}(v)|=\D-1$,  then $|A_{\vp}(v)|\ge h$ by \eqref{eq:A:claim}.
Thus, the lemma holds. 
\end{proof}

For an uncolored vertex $v$ with a $\vp$-available color $c$, an extension $\vp'$ of $\vp$ obtained by coloring  $v$ with $c$ is a {\it good extension} of $\vp$ to $v$ with $c$. 
For a partial good  coloring $\vp$ of $G$ and two adjacent uncolored vertices $u$ and $v$, we say a common neighbor $w$ of $u$ and $v$ is \textit{$\vp$-special} if $w$ has exactly $h+1$ unique colors under $\vp$.

\begin{lemma}\label{lem:nospecial}
For $h\le \D-2$, let  $\vp$ be a partial good $((h+1)\D-1)$-coloring of $G$.
Suppose that two adjacent uncolored vertices $u$ and $v$ have no $\vp$-special neighbor. 
If $\vp'$ is a good extension of $\vp$ to $v$, then $A_{\vp'}(u)\supseteq A_{\vp}(u)\setminus \{ \vp'(v)\}$.
\end{lemma}
\begin{proof}
Since $u$ and $v$ have no $\vp$-special neighbor, a good extension $\vp'$ of $\vp$ to $v$ cannot create a risky neighbor of $u$, so  $\RN_{\vp'}(u)\subseteq \RN_{\vp}(u)$. 
For each $\vp'$-risky neighbor $u'$ of $u$,  
$$\{\text{unique colors in }N_G(u')\text{ under }\vp'\} \subseteq \{\text{unique colors in }N_G(u')\text{ under }\vp\} \cup \{\vp'(v)\},$$
so, $\FC_{\vp'}(u)\subseteq \FC_{\vp}(u)\cup \{\vp'(v)\}$.
By \eqref{eq:A}, 
\[A_{\vp'}(u)=\mathcal{C}\setminus \FC_{\vp'}(u)\supseteq \mathcal{C}\setminus \left(\FC_{\vp}(u)\cup \{\vp'(v)\}\right)=A_{\vp}(u)\setminus \{\vp'(v)\}.\]
%
so the lemma holds. 
\end{proof}

For $j \in [n]$, a sequence $\vp_1,\ldots,\vp_{j}$ of partial good colorings of $G$ is {\it nice} if 
\begin{enumerate}
    \item[(S1)] for $i\in[j]$, $V(\vp_i)=\{v_1,\ldots,v_i\}$, 
    \item[(S2)] for $i\in[j]$, every uncolored vertex $v$ has a $\vp_i$-available color,
    \item[(S3)] for $i\in[j-1]$, $\vp_{i+1}$ is a good extension of $\vp_i$ to $v_{i+1}$, and
    \item[(S4)] for $i\in\{\ell, \ldots, j\}$, if $h=\D-3$ and $\deg_G(v_n)=\D$, then $|\vp_i(N_G(v_n))|=\D-1$.
 \end{enumerate}

When $\psi$ is a partial good coloring in a nice sequence of $G$, every vertex that is uncolored under $\psi$ has a $\psi$-available color, so a good coloring in a nice sequence is in a certain sense stronger than a good coloring of $G$. 
We now proceed to claim that we can color all but three vertices of $G$ in this stronger sense.

\begin{claim}\label{clm:goodsequence}
$G$ has a nice sequence $\vp_1, \ldots, \vp_{n-3}$.
\end{claim}

\begin{proof}
Define $\vp_1$ to be a coloring that assigns an arbitrary color to $v_1$.
Since (S1) and (S2) hold, $\vp_1$ is a nice  sequence of $G$.
For $i \in [n-4]$, suppose $G$ has a nice  sequence $\vp_1, \ldots, \vp_i$.
We will find a good extension $\vp_{i+1}$ of $\vp_i$ to $v_{i+1}$ such that $\vp_1,\ldots, \vp_{i+1}$ is a nice sequence via induction.
Thus, we only need to check (S2) and (S4). 


When $i=\ell-1$ (before going further to obtain $\vp_{i+1}$), 
if every vertex on $E_0$ has degree $\D$, then we  relabel the vertices on $E_0$ in the following way: 
\begin{enumerate}\label{eq:vn} 
\item[(O)] pick a vertex on $E_0$ to be $v_n$ so that $|\vp_{\ell-1}(N_G(v_n))|=\min\{|\vp_{\ell-1}(N_G(x))| \mid x\in V(E_0) \}$.
\end{enumerate}
Note that this new labelling still satisfies Lemma~\ref{lemma:pcf:sequence}.

First, let us check (S4) holds.
Let $h=\Delta-3$ and $\deg_G(v_n)=\Delta$.
Since $\varphi_{\ell-1}$ is good, $|\varphi_{\ell-1}(N_G(v_k))|= \Delta-2=h+1$ for each $v_k \in V(E_0)$.
To see that (S4) holds, we only need to prove that $\varphi_{\ell}(v_{\ell}) \in A_{\varphi_{\ell-1}}(v_{\ell}) \setminus \varphi_{\ell-1}(N_G(v_n))$.
Since $|\varphi_{\ell-1}(N_G(v_k))| =\Delta-2=h+1$ for each $v_k \in V(E_0)$, $v_{\ell+1}$ and $v_n$ are not $\varphi_{\ell-1}$-risky neighbors of $v_{\ell}$.
Also, $v_{\ell+1}$ and $v_n$ are uncolored neighbors of $v_{\ell}$ under $\varphi_{\ell-1}$.
By \eqref{eq:A:claim}, $|A_{\varphi_{\ell-1}}(v_n) \setminus \varphi_{\ell-1}(N_G(v_n))| \ge (2h+1)-(h+1) = h \ge 1$.
Thus, we can pick $\varphi_{\ell}(v_{\ell}) \in A_{\varphi_{\ell-1}}(v_n) \setminus \varphi_{\ell-1}(N_G(v_n))$, and thus (S4) holds.

\begin{subclaim}\label{subclaim:2-8}
There is a good extension $\vp_{i+1}$ of $\vp_i$ to $v_{i+1}$ such that $v_{i+2}$ has a $\vp_{i+1}$-available color.
\end{subclaim}
\begin{proof}
Suppose $\vp'$ is a good extension of $\vp_i$ to $v_{i+1}$ such that $v_{i+2}$ has no $\vp'$-available color. 

We first show that $i=\ell-1$, $h=\D-3$, and $\deg_G(v_n)=\D$ do not all hold.
Otherwise, (O) and (S4) imply  $|\vp_{i}(N_G(v_{i+1}))|=h+1=\D-2$.
Thus, $v_{i+1}\notin \RN_{\vp'}(v_{i+2})$ and $v_{i+2}$ has an uncolored vertex under $\vp'$, and hence $|A_{\vp'}(v_{i+2})|\neq \emptyset$ by \eqref{eq:A:claim}, which is a contradiction.

By Lemma~\ref{lem:uncolored} (i) and (ii), $\deg_G(v_{i+2})=\D$ and $|\vp'(N_G(v_{i+2}))|\ge \D-1$.
Since $i\le n-4$, $v_{i+2}$ has an uncolored neighbor under $\vp'$ by Lemma~\ref{lemma:pcf:sequence} (i), so $|\vp'(N_G(v_{i+2}))|= \D-1$. 
Thus, $v_{i+1}$ and $v_{i+2}$ are adjacent vertices by Lemma~\ref{lemma:pcf:sequence} (ii).
Lemma~\ref{lem:uncolored} (iv) further implies $|\RN_{\vp'}(v_{i+2})|=\D$, so $v_{i+1}\in \RN_{\vp'}(v_{i+2})$.
Thus, $|\vp'(N_G(v_{i+1}))|\le h+\left\lfloor\frac{\D-1-h}{2}\right\rfloor \le   \D-2$.
Since coloring $v_{i+1}$ does not change the colors in the neighborhood of $v_{i+1}$, it follows that $|\vp_i(N_G(v_{i+1}))|\le  \D-2$, so by \eqref{eq:A:claim}, 
\begin{eqnarray}\label{eq:B}
|A_{\vp_{i}}(v_{i+1})|\ge  2+h(\D-|\RN_{\vp_i}(v_{i+1})|)-1=h(\D-|\RN_{\vp_i}(v_{i+1})|)+1. \end{eqnarray}
Since $|\vp'(N_G(v_{i+2}))|= \D-1$, no color appears twice in $N_G(v_{i+2})$.
Thus, when $v_{i+1}$ is uncolored,  $v_{i+2}$ has $\D-2 \ge h+1$ unique colors, so $v_{i+2}\not \in \RN_{\vp_i}(v_{i+1})$. 

\smallskip

\noindent {\bf(Case 1)} Suppose that $v_{i+1}$ and $v_{i+2}$ have no $\vp_{i}$-special neighbor.
By Lemma~\ref{lem:nospecial}, $\emptyset=A_{\vp'}(v_{i+2})\supseteq A_{\vp_{i}}(v_{i+2})\setminus\{\vp'(v_{i+1})\}$.
Thus $A_{\vp_{i}}(v_{i+2})\subseteq \{\vp'(v_{i+1})\}$, yet $A_{\vp_{i}}(v_{i+2})\neq \emptyset$ by the inductive hypothesis, so $A_{\vp_{i}}(v_{i+2})= \{\vp'(v_{i+1})\}$.
Since $|A_{\vp_{i}}(v_{i+1})|\ge h+1\geq 2$ by \eqref{eq:B}, let $\vp_{i+1}$ be a good extension of $\vp_i$ to $v_{i+1}$ with $c'\in A_{\vp_{i}}(v_{i+1})\setminus\{\vp'(v_{i+1})\}$. 
Then $A_{\vp_{i+1}}(v_{i+2})\supseteq A_{\vp_{i}}(v_{i+2})\setminus\{c'\}=\{\vp'(v_{i+1})\}$ by Lemma~\ref{lem:nospecial}. 
\smallskip

\noindent {\bf (Case 2)} Suppose that $v_{i+1}$ and $v_{i+2}$ have a $\vp_{i}$-special neighbor $v$.
Let $X$ be the set of all unique colors in  the neighborhood of $v$ under $\vp_{i}$, so $|X|=h+1$.
Note that $v\notin \RN_{\vp_{i}}(v_{i+1})$, so two neighbors of $v_{i+1}$ are not $\vp_i$-risky neighbors. 
Since $|A_{\vp_{i}}(v_{i+1})\setminus X|\ge  (2h+1)   -(h+1) \ge h$ by \eqref{eq:B}, let $\vp_{i+1}$ be a good extension of $\vp_i$ to $v_{i+1}$ with $c'\in A_{\vp_{i}}(v_{i+1})\setminus X$.    
Then every color in $X$ is a unique color in $N_G(v)$ under $\vp_{i+1}$, so $v\notin \RN_{\vp_{i+1}}(v_{i+2})$.
Since $v_{i+2}$ has an uncolored neighbor under $\vp_{i+1}$ 
and $|\RN_{\vp_{i+1}}(v_{i+2})|\leq \D-1$, by Lemma~\ref{lem:uncolored} (iv), $v_{i+2}$ has a $\vp_{i+1}$-available color.
\end{proof}

By Subclaim~\ref{subclaim:2-8}, there is a good extension $\vp_{i+1}$ of $\vp_i$ to $v_{i+1}$ such that $v_{i+2}$ has a $\vp_{i+1}$-available color. 
We now show that each of $v_{i+3}, \ldots, v_n$ has a $\vp_{i+1}$-available color.

If $v_k$ is a vertex where $i+3\le k\le n-1$, 
then Lemma~\ref{lemma:pcf:sequence} (ii) implies $v_k$ has at least two uncolored neighbors under $\vp_{i+1}$, so $v_k$ has at most $\D-2$ colors in its neighborhood  under $\vp_{i+1}$.
Lemma~\ref{lem:uncolored} (ii) implies $v_k$ has a $\vp_{i+1}$-available color.
It remains to confirm that $v_n$ has a $\vp_{i+1}$-available color.

Suppose to the contrary that $v_n$ has no $\vp_{i+1}$-available color. 
By Lemma~\ref{lem:uncolored} (i) and (ii),  $\deg_G(v_n)=\D$ and $|\vp_{i+1}(N_G(v_n))|\ge \D-1$. 
Since $v_n$ is a vertex with minimum degree among $V(E_0)$, this implies all vertices on $E_0$ have degree $\D$.
Since $i\le n-4$, we know $v_n$ has an uncolored neighbor $v_{n-1}$ under $\vp_{i+1}$, so $|\vp_{i+1}(N_G(v_n))|=\D-1$. 
Thus, $v_\ell$ is colored under $\vp_{i+1}$,
so $i\ge \ell-1$ and $|\vp_{\ell-1}(N_G(v_n))|=\D-2$. 
By  (O), 
$|\vp_{\ell-1}(N_G(v_{n-1}))|=\D-2$ since $v_{n-2}$ and $v_n$ are uncolored under $\vp_{\ell-1}$.
Furthermore, since $i \le n-4$, $|\vp_{i+1}(N_G(v_{n-1}))|= \D-2$.
On the other hand, $|\RN_{\vp_{i+1}}(v_n)|=\D$ by Lemma~\ref{lem:uncolored} (iv),  so $v_{n-1} \in \RN_{\vp_{i+1}}(v_n)$. 
Then $\D-2 \le h$, which is a contradiction to the assumption that $\D\geq h+3$. 
Thus, (S2) holds.
\end{proof}

Let $\vp_1, \ldots, \vp_{n-3}$ be a nice sequence of $G$, guaranteed to exist by the previous claim.
Note that only three vertices $v_{n-2}$, $v_{n-1}$, $v_n$ are uncolored under $\vp_{n-3}$, and each of the three vertices has a  $\vp_{n-3}$-available color. For two uncolored vertices $u$ and $v$, we say $u$ is \textit{$(\vp, v)$-safe} if $u$ has a $\vp'$-available color for every good extension $\vp'$ of $\vp$ to $v$.

\begin{claim}\label{claim:triangle}
$E_0$ is not a $3$-cycle. 
\end{claim}

\begin{proof} 
Suppose to the contrary that $E_0$ is a $3$-cycle. 
Thus, $V(E_0)=\{v_{n-2}, v_{n-1}, v_n\}$. 

\begin{subclaim}\label{subclm:safe}
For two vertices $x,y$ on $E_0$, if $|\vp_{n-3}(N_G(x))|\le |\vp_{n-3}(N_G(y))|$, then $x$ is $(\vp_{n-3}, y)$-safe.
\end{subclaim}
\begin{proof}
Suppose to the contrary that $x$ has no $\vp$-available color for some good extension $\vp$ of $\vp_{n-3}$ to $y$.
By Lemma~\ref{lem:uncolored} (ii), $|\vp(N_G(x))|\ge \D-1$.
Yet, $x$ has an uncolored neighbor under $\vp$, so $|\vp(N_G(x))|= \D-1$, which further implies $|\vp_{n-3}(N_G(x))|=\D-2$.
Thus, $|\vp(N_G(y))|=|\vp_{n-3}(N_G(y))| \ge |\vp_{n-3}(N_G(x))| = \D-2 \ge h+1$.
Hence, $y\notin \RN_{\vp}(x)$, so $|\vp(N_G(x))|+|\RN_\vp(x)|\leq 2\D-2$, which is a contradiction to Lemma~\ref{lem:uncolored} (iv).
\end{proof}

Now, relabel the vertices on $E_0$ so that $|\vp_{n-3}(N_G(v_n))|\le |\vp_{n-3}(N_G(v_{n-1}))| \le |\vp_{n-3}(N_G(v_{n-2}))|$.
Assume $\vp$ is a good extension of $\vp_{n-3}$ to $v_{n-2}$. 
By Subclaim~\ref{subclm:safe},
$A_{\vp}(v_{n-1})\neq \emptyset$ and  $A_{\vp}(v_{n})\neq \emptyset$. 
Moreover, since $G$ is a counterexample to the theorem, $v_n$ (resp. $v_{n-1}$) has no $\vp'$-available color  for every extension $\vp'$ of $\vp$ to $v_{n-1}$  (resp. $v_{n}$) obtained by coloring  $v_{n-1}$  (resp. $v_{n}$) with a color in  $A_{\vp}(v_{n-1})$  (resp. $A_{\vp}(v_{n})$).
By Lemma~\ref{lem:uncolored} (i) and (ii), for $w\in\{v_{n-1}, v_n\}$, \begin{eqnarray}\label{vn;eq}
&&\deg_G(w)=\D \quad \text{ and }\quad |\vp(N_G(w))|\ge \D-2.
\end{eqnarray} 

\begin{subclaim} \label{subclm:extreme}
  $|\vp_{n-3}(N_G(w))|=\D-2\ge h+2$ for each $w$ on $E_0$.
\end{subclaim}
\begin{proof}
First, suppose to the contrary that $h=\D-3$.
If $|\vp_{n-3}(N_G(v_{n}))| \le \D-3$, then there is a repeated color on $N_G(v_n)$  under $\vp_{n-3}$ since $v_n$ has $\D-2$ colored neighbors under $\vp_{n-3}$.
Note that $\deg_G(v_n)=\D$ by \eqref{vn;eq}.
Since there is a repeated color on $N_G(v_n)$ under $\vp_{n-3}$ and there are at least $h$ unique colors on $N_G(v_n)$,
we know $h+1\le |\vp_{n-3}(N_G(v_n))|$, which implies $h\le \D-4$, a contradiction. 
Hence $|\vp_{n-3}(N_G(v_i))|=\D-2=h+1$ for each $i\in\{n-2, n-1, n\}$. 
Thus neither $v_{n-1}$ nor $v_n$ is a $\vp_{n-3}$-risky neighbor of $v_{n-2}$. 
Moreover, since $|A_{\vp_{n-3}}(v_{n-2})\setminus \vp_{n-3}(N_G(v_{n-1}))|\ge (2+2h-1)-(h+1)=h\ge 1$ by~\eqref{eq:A:claim}, we may assume that $\vp$ was obtained by coloring $v_{n-2}$ with a color in $A_{\vp_{n-3}}(v_{n-2})\setminus \vp_{n-3}(N_G(v_{n-1}))$, so $|\vp(N_G(v_{n-1}))|=h+2$.
Let $\vp^1$ be a good extension of $\vp$ to $v_{n-1}$, so $|\vp^1(N_G(v_{n-1}))|= h+2$.
Therefore, $v_{n-1}\not\in \RN_{\vp^1}(v_{n})$.
Since $G$ is a counterexample to the theorem, $v_n$ has no $\vp^1$-available color, so Lemma~\ref{lem:uncolored}~(iv) implies $|\vp^1(N_G(v_n))|=\D$ and 
$v_{n-2}\in \RN_{\vp^1}(v_{n})$. 
Thus $\vp^1(v_{n-1})\in \vp_{n-3}(N_G(v_{n-2}))$.
Let $\vp^{2}$ be the coloring obtained from $\vp^1$ by deleting the color on $v_{n-2}$.
Now, color $v_{n}$ with a $\vp^2$-available color to obtain $\vp^3$; this is possible since $v_{n}$ has an uncolored neighbor $v_{n-2}$ under $\vp^2$ and $v_{n-1}\not\in \RN_{\vp^2}(v_n)$ because $|\vp^2(N_G(v_{n-1}))|\geq h+1$.
At this point all vertices except $v_{n-2}$ are colored and $|\vp^{3}(N_G(v_{n-2}))|\le \D-1$ since $\vp^1(v_{n-1})\in \vp_{n-3}(N_G(v_{n-2}))$.
Since $|\vp^1(N_G(v_n))|=\D$, we know $v_n$ has $\D-1$ unique colors under $\vp^{3}$, so $v_n\not\in  \RN_{\vp^{3}}(v_{n-2})$ because $h=\D-3$. 
Then Lemma~\ref{lem:uncolored} (iv) implies $v_{n-2}$ has a $\vp^3$-available color, which we can assign to $v_{n-2}$ to color all vertices of $G$. 
This is a contradiction. 

Now we know $h\le \D-4$.
Let $\varphi$ be a good extension of $\varphi_{n-3}$ to $v_{n-2}$.
Since $A_{\vp}(v_{n-1}) \neq \emptyset$, 
there is a good extension $\vp'$ of $\vp$ to $v_{n-1}$.
Then by \eqref{vn;eq}, $|\vp(N_G(v_{n-1}))|\ge \D-2 \ge h+2$, so $v_{n-1}$ has at least $h+1$ unique colors under $\vp'$.
Thus, $v_{n-1}$ is not a $\vp'$-risky neighbor of $v_n$, so $|\RN_{\vp'}(v_n)|\le \D-1$.
If $|\vp_{n-3}(N_G(v_n))| \le \D-3$, then $|\vp'(N_G(v_n))|\le \D-1$.
By Lemma~\ref{lem:uncolored} (iv), 
$v_n$ has a $\vp'$-available color, which we can assign to $v_n$ to color all vertices of $G$.
This is a contradiction, so $|\vp_{n-3}(N_G(v_n))| =\D-2$.
Since $|\vp_{n-3}(N_G(v_n))| \le |\vp_{n-3}(N_G(v_{n-1}))| \le |\vp_{n-3}(N_G(v_{n-2}))|$, 
and both $v_{n-1}$ and $v_{n-2}$ have two uncolored neighbors under $\vp_{n-3}$, the subclaim is proved.
\end{proof}

For a good extension $\vp'$ of $\vp$ to $v_{n-1}$,
Subclaim~\ref{subclm:extreme} implies $v_{n-2}, v_{n-1}\not\in \RN_{\vp'}(v_n)$.
Thus, $v_n$ has a $\vp'$-available color by Lemma~\ref{lem:uncolored} (iii), which is a contradiction to the assumption that $G$ is a counterexample. 
\end{proof}
 
Since $E_0$ is a shortest cycle of $G$, Claim~\ref{claim:triangle} implies $G$ has no $3$-cycles.
Thus, $v_{n-2}$ and $v_n$ are not adjacent to each other.
Moreover, special neighbors do not exist.

We now show a series of claims showing $h=\D-3$ and $E_0$ cannot be a $5$-cycle. 
Recall that if $h=\D-3$ and $\deg_G(v_n)=\D$, then $|\vp_{n-3}(N_G(v_n))|=\D-1$ by (S4).

\begin{claim}\label{claim:triangle2}
The following hold:
\begin{enumerate}[(i)]
\item Every vertex $v_k$ on $E_0$ has degree $\D$ and has $\D-2 \ge  h+1$ unique colors under $\vp_{\ell-1}$. 
\item  Both $v_{n-2}$ and $v_n$ have at least $h+2$ unique colors under $\vp_{n-3}$. 
\end{enumerate}
\end{claim} 
\begin{proof}
We will first show that $|\vp_{n-3}(N_G(v_i))| = \D -1$ for $i \in \{n-2, n\}$.
Suppose to the contrary that $|\vp_{n-3}(N_G(v_i))|\le \D-2$ for some $i\in \{n-2,n\}$, and let $j\in \{n-2,n\}\setminus\{i\}$.
Let $\vp$ be a good extension of $\vp_{n-3}$ to $v_j$, so $\vp(N_G(v_i))=\vp_{n-3}(N_G(v_i))$ since $v_i$ and $v_j$ are not adjacent to each other.
Thus, $|\vp(N_G(v_i))|\le \D-2$, so Lemma~\ref{lem:uncolored} (ii)  implies $v_i$ has a $\vp$-available color. 
Thus $v_i$ is $(\vp_{n-3}, v_j)$-safe. By applying Lemma~\ref{lem:nospecial} to $v_{n-1}$ and $v_j$, $A_{\vp}(v_{n-1})\supseteq A_{\vp_{n-3}}(v_{n-1}) \setminus \{\vp(v_j)\}$.

Suppose $v_{n-1}$ has no $\vp$-available color, which implies $A_{\vp_{n-3}}(v_{n-1})=\{\vp(v_j)\}$. 
By Lemma~\ref{lem:uncolored} (ii), $|\vp(N_G(v_{n-1}))|\ge \D-1$. 
Since $v_{n-1}$ has an uncolored neighbor $v_i$ under $\vp$, $|\vp(N_G(v_{n-1}))|= \D-1\geq h+2$, so $v_{n-1}\notin \RN_{\vp_{n-3}}(v_j)$.
Lemma~\ref{lem:uncolored}~(iv) implies $|\RN_{\vp}(v_{n-1})|=\D$.
Since $v_j\in \RN_{\vp_{n-3}}(v_{n-1})$, $|\vp_{n-3}(N_G(v_j))|\le \D-2$.
By \eqref{eq:A:claim}, $|A_{\vp_{n-3}}(v_j)|\ge h+1\ge 2$, so there is a good extension $\vp'$ of $\vp_{n-3}$ to $v_j$ such that $v_{n-1}$ has a $\vp'$-available color.
By renaming if necessary, assume $\vp$ is a good extension of $\vp_{n-3}$ to $v_j$ so that $A_{\vp}(v_{n-1})\neq \emptyset$.

Let $\vp^*$ be a good extension of $\vp$ to $v_{n-1}$, and let $\alpha=\vp^*(v_{n-1})$.
By applying Lemma~\ref{lem:nospecial} to $v_i$ and $v_{n-1}$, it follows that $ A_{\vp^*}(v_i) \supseteq A_{\vp}(v_i)\setminus \{\alpha\}$. 
This implies $A_{\vp}(v_i)= \{\alpha\}$. 
Let $\vp^{**}$ be a good extension of $\vp$ to $v_{i}$ with $\alpha$.
By applying Lemma~\ref{lem:nospecial} to $v_i$ and $v_{n-1}$, it follows that $ A_{\vp^{**}}(v_{n-1}) \supset A_{\vp}(v_{n-1})\setminus \{\vp^{**}(v_{i})\}$. This implies that $A_{\vp}(v_{n-1})= \{\alpha\}$.

By applying Lemma~\ref{lem:uncolored} (i) and (ii) to $\vp^*$ and $\vp^{**}$, for each $k\in\{n-1,i\}$, we know $\deg_G(v_{k})=\D$ and $|\vp(N_G(v_{k}))|\ge \D-2$.
Also, since $|\vp(N_G(v_i))|\le \D-2$ and $|A_{\vp}(v_{i})|=1$, it follows that $|\RN_{\vp}(v_{i})|=\D$ by \eqref{eq:A:claim}.
In addition, since $v_{n-1}\in \RN_{\vp}(v_{i})$, we know $|\vp(N_G(v_{n-1}))|=\D-2$, so $|\RN_{\vp}(v_{n-1})|=\D$ by \eqref{eq:A:claim} and the fact that $|A_{\vp}(v_{n-1})|=1$.
If $h\le \D-4$, then $|\vp(N_G(v_{n-1}))| \ge \D-2\ge h+2$,  so $v_{n-1}$ is not a $\vp$-risky neighbor of $v_{i}$, which is a contradiction.
Thus $h=\D-3$. 
Note that $|\vp_{n-3}(N_G(v_{i}))|\le \D-2$ by the choice of $i$.
If $|\vp_{n-3}(N_G(v_{n}))|\le \D-2$, then we may set $i=n$, and therefore $\deg_G(v_i)=\deg_G(v_n)=\Delta$ and $|\vp_{n-3}(N_G(v_n))|=\D-1$ by (S4), which is a contradiction.
Thus $|\vp_{n-3}(N_G(v_n))|\ge \D-1=h+2$, which is a contradiction to the fact that $v_n$ is a $\vp$-risky neighbor of $v_{n-1}$.
Thus $|\vp_{n-3}(N_G(v_i))| = \D -1$ for $i\in\{n-2, n\}$.


Since $v_n$ has an uncolored neighbor $v_{n-1}$ under $\vp_{n-3}$, we know $v_n$ has degree $\D$.
Moreover, $|\vp_{\ell-1}(N_G(v_{n}))| = \D-2$.
Every vertex $v_k$ on $E_0$ has degree $\D$ by the choice of $v_n$ and $|\vp_{\ell-1}(N_G(v_k))| = \D-2$ by (O), so (i) holds.

For $i\in\{n-2, n\}$, since $v_{n-1}$ is an uncolored neighbor of $v_i$ under $\vp_{n-3}$, we conclude $v_i$ has at least $h+2$ unique colors under $\vp_{n-3}$, so  (ii) holds. 
\end{proof}

For simplicity, let $X$ be the set of colors on $N_{G}(v_{n-1})$ under $\vp_{n-3}$, and note that $|X| = \D-2\ge  h+1$ by Claim~\ref{claim:triangle2}~(i).

\begin{claim}\label{last:claim}
The following hold:
\begin{enumerate}[(i)]
\item $|\vp_{n-3}(N_G(v_{n-3}))| = \D-2$.
\item $h=\D-3$.
\item $A_{\vp_{n-3}}(v_{n-2})$ and $A_{\vp_{n-3}}(v_{n})$ are disjoint subsets of $X$.
\end{enumerate}
\end{claim}
\begin{proof}
Let $\vp$ be a good extension of $\vp_{n-3}$ to $v_{n}$. 
If $v_{n-2}$ has a $\vp$-available color $c$, then let $\vp'$ be the (good) extension of $\vp$ to $v_{n-2}$ with $c$. 
By Claim~\ref{claim:triangle2}~(ii), both $v_{n-2}$ and $v_{n}$ have at least $h+2$ unique colors under $\vp_{n-3}$. 
Since $G$ has no $3$-cycle, $v_{n-2}$ and $v_n$ are not adjacent to each other, so both $v_{n-2}$ and $v_n$ still have at least $h+2$ unique colors under  $\vp'$.
Thus, $v_{n-2}, v_{n}\notin \RN_{\vp'}(v_{n-1})$, so by Lemma~\ref{lem:uncolored} (iii), $v_{n-1}$ has a $\vp'$-available color.
This is a coloring of all of $G$, which is a contradiction, hence, $v_{n-2}$ has no $\vp$-available color. 
Furthermore, since $v_{n-2}$ has an uncolored neighbor under $\vp$, every neighbor of $v_{n-2}$ must be a $\vp$-risky neighbor by Lemma~\ref{lem:uncolored} (iv).

In particular, $v_{n-3}$ is a $\vp$-risky neighbor of $v_{n-2}$.
By Claim~\ref{claim:triangle2}~(i), $v_{n-3}$ has at least $h+1$ unique colors in its neighborhood under $\vp_{\ell-1}$, 
so $\vp_{n-3}(v_{n-4}) \in \vp_{\ell-1}(N_G(v_{n-3}))$ if $v_{n-4} \in N_G(v_{n-3})$.
Thus, $|\vp_{n-3}(N_G(v_{n-3}))| = \D-2$, so (i) holds.

Note that $v_{n-1}$ is not a $\vp_{n-3}$-risky vertex by Claim~\ref{claim:triangle2}~(i).
Thus, $v_{n-1}$ becomes a $\vp$-risky vertex when coloring $v_{n}$, so $\vp(v_{n})\in \vp_{n-3}(N_G(v_{n-1}))$.
Hence, $v_{n-1}$ has exactly $\D-3$ unique colors under $\vp$, so $\D-3 \le h$.
Therefore, $h=\D-3$, so (ii) holds. 
Also, $A_{\vp_{n-3}}(v_{n})\subseteq X$, since $\vp(v_{n}) \in \vp_{n-3}(N_G(v_{n-1})) = X$ and $\vp(v_{n})$ is an arbitrary element in $A_{\vp_{n-3}}(v_{n})$. By symmetry,  $A_{\vp_{n-3}}(v_{n-2})\subseteq X$.

If $c\in A_{\vp_{n-3}}(v_{n-2})\cap A_{\vp_{n-3}}(v_{n})$, then let $\vp'$ be an extension of $\vp_{n-3}$ by coloring both $v_{n-2}$ and $v_n$ with $c$. 
Note that $\vp'$ is a good extension, since if $u$ is a common neighbor of $v_{n-2}$ and $v_n$, then $\deg_G(u)=\D$ by the choice of cycle $E_0$, and therefore $u$ has $h+1$ unique colors under $\varphi_{n-3}$, which implies that $c \in A_{\vp}(v_n)$ for an extension $\vp$ of $\vp_{n-3}$ by assigning a color $c$ to $v_{n-2}$. 
By Claim~\ref{claim:triangle2}~(ii), both $v_{n-2}$ and $v_{n}$ have at least $h+2$ unique colors, so neither $v_{n-2}$ nor $v_n$ is a $\vp'$-risky neighbor of $v_{n-1}$. 
Thus, $v_{n-1}$ has a $\vp'$-available color, which we can use on $v_{n-1}$ to obtain a coloring of all vertices of $G$.
This is a contradiction.
Hence,  $A_{\vp_{n-3}}(v_{n-2})$ and $A_{\vp_{n-3}}(v_{n})$ are disjoint subsets of $X$, so (ii) holds. 
\end{proof}

For $i \in \{n-2,n\}$, since $v_{n-1}$ is both uncolored and not a $\vp_{n-3}$-risky neighbor of $v_i$, we know $|A_{\vp_{n-3}}(v_{i})|\ge 1+h-1=h$ by \eqref{eq:A:claim}.
By Claim~\ref{last:claim}~(iii),
$2h  \le |A_{\vp_{n-3}}(v_{n-2})\cup A_{\vp_{n-3}}(v_{n})| \le |X|=h+1$, 
 so $h=1$, which further implies $\D=4$.
Thus, we may assume $A_{\vp_{n-3}}(v_{n-2})=\{\al\}$, $A_{\vp_{n-3}}(v_{n})=\{\be\}$, and $X=\{\al,\be\}$. 

\begin{claim}\label{claim:vp^*+n-4}
The following hold:
\begin{enumerate}[(i)]
\item
For every good extension $\vp^*_{n-3}$ of $\vp_{n-4}$ to  $v_{n-3}$, $\vp_{n-3}^*(N_{G}(v_{n-1}))=\{\al,\be\}$ and $\vp_1$,$\ldots$,$\vp_{n-4}$,$\vp_{n-3}^*$ is a nice sequence.
\item 
$A_{\vp_{n-4}}(v_{n-3})=\{\al,\be\}$. 
\item 
$E_0$ is not a $5$-cycle. 
\end{enumerate}
\end{claim}
\begin{proof}
(i)
Let $\vp^*_{n-3}$ be a good extension   of $\vp_{n-4}$ to  $v_{n-3}$.
Since $G$ has no $3$-cycles, $v_{n-1}$ is not adjacent to $v_{n-3}$.
Thus, $\vp_{n-3}(N_{G}(v_{n-1}))=X=\vp_{n-3}^*(N_{G}(v_{n-1}))$.
Since $v_{n-1}$ has two uncolored neighbors under $\vp_{n-3}^*$, $v_{n-1}$ has a $\vp_{n-3}^*$-available color  by Lemma~\ref{lem:uncolored} (ii).
Moreover, for $i\in\{n-2,n\}$, $v_{n-1}$ is not a $\vp^*_{n-3}$-risky neighbor of $v_{i}$ by Claim~\ref{claim:triangle2}~(i), and $v_{n-1}$ is an uncolored neighbor of $v_i$.
Thus, $v_i$ has a $\vp_{n-3}^*$-available color  by Lemma~\ref{lem:uncolored} (iv).
Thus $\vp_1,\ldots,\vp_{n-4}, \vp_{n-3}^{*}$ is a nice sequence. 

(ii)
Recall $\D=4$ and $h=1$.
Claim~\ref{last:claim}~(i) implies $|\vp_{n-4}(N_G(v_{n-3}))|=|\vp_{n-3}(N_G(v_{n-3}))| = 2$.

By Claim~\ref{claim:triangle2}~(i),
$v_{n-2}$ is not a $\vp_{n-4}$-risky neighbor of $v_{n-3}$.
Thus, by \eqref{eq:A:claim},   
$|A_{\vp_{n-4}}(v_{n-3})|\ge 2+h-1= 2$. 

Let $\vp^*_{n-3}$ be a good extension of $\vp_{n-4}$ to   $v_{n-3}$ with  $c^*\in A_{\vp_{n-4}}(v_{n-3})$.
(i) and Claim~\ref{last:claim}~(iii) imply that  $A_{\vp^*_{n-3}}(v_{n-2})=\{c\}$ for some $c\in\{\al,\be\}$.
By Lemma~\ref{lem:nospecial}, $A_{\vp^*_{n-3}}(v_{n-2})$ contains $A_{\vp_{n-4}}(v_{n-2})\setminus\{ c^*\}$. 
Claim~\ref{claim:triangle2}~(i) implies $v_{n-1}$ is not a $\vp_{n-4}$-risky neighbor of $v_{n-2}$. 
Since $v_{n-3}$ and $v_{n-1}$ are uncolored neighbors of $v_{n-2}$ under $\vp_{n-4}$, we know  $|A_{\vp_{n-4}}(v_{n-2})|\ge 2$ by 
(\ref{eq:A:claim}).
Thus $A_{\vp_{n-4}}(v_{n-2}) = \{c,c^*\}$.
Since $|A_{\vp_{n-4}}(v_{n-3})|\ge 2$, there are at least two ways to extend $\vp_{n-4}$ to $\vp^*_{n-3}$. 
As the choice of $c^*$ was arbitrary, it follows that $A_{\vp_{n-4}}(v_{n-3})=\{\al,\be\}$.

(iii)
If $E_0$ is a $5$-cycle, then $v_n$ and $v_{n-3}$ have a  common neighbor $v_{n-4}$. 
By Claim~\ref{claim:triangle2}~(i), $v_{n-2}$ and $v_{n-4}$ are not $\vp_{n-4}$-risky neighbors of $v_{n-3}$. 
Also, Claim~\ref{last:claim}~(i) implies $|\vp_{n-4}(N_G(v_{n-3}))|=|\vp_{n-3}(N_G(v_{n-3}))| = 2$.
Thus, by \eqref{eq:A:claim}, $|A_{\vp_{n-4}}(v_{n-3})|\geq 2+2h-1\geq 3$, which is a contradiction to (ii).
\end{proof}

By Claim~\ref{claim:vp^*+n-4} (ii), there are two ways to obtain $\vp_{n-3}$ from $\vp_{n-4}$; in other words, $\vp_{n-3}$ is a good extension of $\vp_{n-4}$ to $v_{n-3}$ with $\vp_{n-3}(v_{n-3})\in\{\al, \be\}$.
By Claim~\ref{last:claim}~(iii), the choice of $\vp_{n-3}(v_{n-3})$ affects the available color of $v_n$, which is possible only when the distance between $v_{n-3}$ and $v_n$ is at most $2$.
Hence, $|V(E_0)|\le 5$. 

Since $A_{\vp_{n-3}}(v_{n-2}) = \{\alpha\}$, Claim~\ref{claim:vp^*+n-4} (ii) implies $\vp_{n-3}(v_{n-3})=\be$.
Since $A_{\vp_{n-3}}(v_{n})=\{\be\}$, 
we know $v_n$ and $v_{n-3}$ cannot be adjacent to each other. 
Hence, $E_0$ is a $5$-cycle, which is a contradiction to Claim~\ref{claim:vp^*+n-4}~(iii).

\section{Proofs of Theorems~\ref{thm:chordal} $\sim$ \ref{thm:claw} } \label{sec:lcc} 

In this section, we prove Theorems~\ref{thm:chordal},~\ref{thm:lcc},~\ref{thm:genclaw},~and~\ref{thm:claw}. 
We first prove results regarding chordal graphs.
A vertex is called \textit{simplicial} if its neighbors form a clique.
Recall that for a chordal graph $G$, there is a \textit{simplicial ordering} $v_1,\ldots,v_n$ of $G$ such that for each $i\in[n]$, $v_i$ and its neighbors in $G[\{v_i,\ldots, v_n\}]$ form a clique. 
The maximal clique containing $v_i$ in the subgraph of $G$ induced by $\{v_i,v_{i+1},\ldots,v_n\}$ is  a \textit{simplicial clique} of $G$. 
Let $s(G)$ denote the maximum size of all possible simplicial cliques of $G$.
Note that $s(G) \le \D(G)+1$ for every chordal graph $G$.
 
\begin{thmm:chordal}
For a positive integer $h$, if $G$ is a chordal graph, then $$\hpcf(G) \le 1 + (h+1) \cdot \min \left\{  s(G)-1,  \frac{\D(G)+h-1}{2}\right\}.$$
\end{thmm:chordal}
\begin{proof}
For simplicity, let $m_G=\min \left\{  s(G)-1, \frac{\D(G)+h-1}{2}\right\}$
and $c_G= \lfloor 1 + (h+1) m_G \rfloor$.
Let $G$ be a minimal chordal graph with no proper $h$-CF $c_G$-coloring.
Choose a simplicial vertex $v$, and let $G'=G-v$. 
Note that $G'$ is also a chordal graph, so $G'$ has a proper $h$-CF $c_G$-coloring $\vp$ since $s(G')\le s(G)$ and $\D(G')\le \D(G)$.
If $\deg_G(v) \le m_G$, then 
color $v$ with a color not in 
$\vp(N_G(v)) \cup \uc(N_G(v))$, which is always possible since
$|N_G(v) \cup \uc(N_G(v))| \le (h+1)|N_G(v)|\le (h+1) \lfloor m_G \rfloor\le \lfloor (h+1) m_G \rfloor  = c_G-1$. 
Then $\vp$ is a proper $h$-CF $c_G$-coloring of $G$, which is a contradiction.

Suppose that $\deg_G(v) > m_G$. 
By definition $s(G)-1\geq \deg_G(v) $, so
 $m_G= \frac{\D(G)+h-1}{2}$. 
Now extend $\vp$ to $v$ with a color not in $\vp(N_G(v))$, which is possible since 
$\deg_G(v) \le \Delta(G)=2m_G-(h-1)\le (h+1)m_G-(h-1) \le \lfloor (h+1) m_G \rfloor= c_G-1$.
Take a neighbor $u$ of  $v$.
Since $v$ is a simplicial vertex, $N_G(u)$ contains a clique of size $\deg_G(v)$.
Thus, $u$ has at least $\deg_G(v) - (\deg_G(u)-\deg_G(v))$ unique colors.
Since $\deg_G(v)- (\deg_G(u)-\deg_G(v)) =2\deg_G(v)-\deg_G(u)>\D(G)+h-1-\deg_G(u)\ge h-1$, the existence of $h$ unique colors in the neighborhood of $u$ is guaranteed.
Thus $\vp$ is a proper $h$-CF $c_G$-coloring of $G$, which is a contradiction.
\end{proof}

Theorem~\ref{thm:chordal} implies that every tree has a proper $h$-CF $(h+2)$-coloring.

Recall that the \textit{local vertex clique cover number}  of a graph $G$, denoted $\lcc(G)$, is the minimum $q$ such that for every vertex $v$, there are at most $q$ cliques whose union is $N_G(v)$.
 For a (partial) coloring $\vp$ of $G$, and a vertex $v$ in $G$, let $\vp_o(v)$ be a color that appears an odd number of times in $\vp(N_G(v))$ if it exists.
We call $\vp_o(v)$ an {\it odd color} of $v$. 
We state some useful lemmas first.

\begin{lemma}\label{lem:bigclique}
Let $K$ be a clique of a graph $G$.
If $G-K$ is odd $m$-colorable for some integer $m\ge 2\D(G)-|K|+3$, then $G$ is odd $m$-colorable.
\end{lemma}
 
\begin{proof}
Since every graph $G$ satisfies $\odd(G)\leq\pcf(G)\leq 2\D(G)+1$ (see  introduction), the lemma is true when $|K|\leq 2$. 
Suppose $|K|\ge 3$, and let $\varphi$ be an odd $m$-coloring  of $H$, where  $H=G-K$. 
For each vertex $u\in K$, let $L(u)$ be the set of colors not in $\varphi(N_G(u)\cap V(H))\cup \vp_o(N_G(u)\cap V(H))$,
where $\vp_o(N_G(u) \cap V(H))= \bigcup_{w \in N_G(u) \cap V(H)} \vp_o(w)$.
Note that each $u\in K$ has at most $\D(G)-|K|+1$ neighbors in $N_G(u)\cap V(H)$, so \[|L(u)|\ge m - 2|N_G(u)\cap V(H)|\ge   |K|+1.\] 
 
Let $v_1,\dots, v_{|K|}$ be the vertices of $K$.
We will assign colors to vertices of $K$, one by one in order as follows: 
 choose a color $\vp(v_1)$ for $v_1$ from $L'(v_1)$ where $L'(v_1)\subseteq L(v_1)$ and $|L'(v_1)|=|K|+1$, and for each $i \in \{2, \ldots, |K|\}$,  choose a color $\vp(v_i)$ for $v_i$ from $L'(v_i)$ where $L'(v_i)\subseteq L(v_i) \setminus \{\vp(v_1),\ldots,\vp(v_{i-1})\}$ and $|L'(v_i)|=|K|+2-i$.
 Since there are exactly $|K|+2-i$ choices for each $v_i$, there are $(|K|+1)!$ colorings for $v_1, \ldots, v_{|K|}$.
 These colorings are proper $L$-colorings, so if we can guarantee odd colors of the vertices of $K$, then they become odd $m$-colorings of $G$.
Let us call these colorings {\it $K$-colorings} of $G$.

For each $u\in K$, let $O(u)$ be the set of odd colors of $u$ on vertices in $N_G(u)\cap V(H)$. 
Note that for $u\in K$, if $|O(u)|\neq |K|-1$, then any $K$-coloring of $G$ defines an odd color for $u$ since $K$ is a clique.
If $|O(v_i)|=|K|-1$ for some $i \in \{1,\ldots,|K|\}$, then we exclude at most $(|K|+2-i)(|K|-1)!$ choices of $K$-colorings such that the {set of} colors for $K\setminus\{v_i\}$ equals $O(v_i)$.
Since 
\[(|K|+1)!-\sum_{i=1}^{|K|}(|K|+2-i)(|K|-1)! = \frac{1}{2}|K|(|K|-1)(|K|-1)! = \frac{1}{2}(|K|-1)|K|!>0 \]
there is a $K$-coloring that defines odd colors for all vertices of $K$.
Hence, $G$ has an odd $m$-coloring.
\end{proof}

\begin{lemma}\label{lem:bigclique2}
For two adjacent vertices $v_1$ and $v_2$ of a graph $G$ with $\D(G)\geq 3$, let $K=N_G[v_1]\cap N_G[v_2]$.
If $G-\{v_1,v_2\}$ is odd $m$-colorable for some integer $m \ge 2\D(G)-|K|+ \left\lceil \frac{3}{2}+\sqrt{2|K|-3}\right\rceil$, then $G$ is odd $m$-colorable.
\end{lemma}

\begin{proof} 
Note that $|K| \ge 2$.
Let $\vp$ be an odd $m$-coloring of $H$ where  $H=G-\{v_1, v_2\}$. 
For $i \in \{1,2\}$, let $L(v_i)$ be the set of colors not in $\vp(N_G(v_i) \setminus \{v_{3-i}\}) \cup \vp_o((N_G(v_i) \setminus K) \cup \{v_{3-i}\})$.
Since $|N_G(v_i) \setminus \{v_{3-i}\}| \le \D(G)-1$ and $|N_G(v_i) \setminus K| \le \D(G)-|K|+1$,
\[|L(v_i)| \ge m-|N_G(v_i) \setminus \{v_{3-i}\}| - |N_G(v_i) \setminus K| -1 \ge m  - 2\D(G) + |K| -1\geq \left\lceil \frac{3}{2}+\sqrt{2|K|-3}\right\rceil-1. 
\]
Note that  $|L(v_i)| \ge 2$ since $|K|\geq 2$. 

We will show that there exists a proper extension $\vp'$ of $\vp$ to $v_1$ and $v_2$ where $\vp'(v_1)\neq \vp'(v_2)$, $\vp'(v_1)\in L(v_1)$, $\vp'(v_2)\in L(v_2)$, and each vertex in $K \setminus \{v_1,v_2\}$ is guaranteed an odd color;
this is an odd coloring of $G$.
For each $u \in K \setminus \{v_1, v_2\}$, let $O(u)$ be the set of odd colors of $u$ on vertices in $N_G(u) \cap V(H)$.
Unless $O(u)=\{\vp'(v_1), \vp'(v_2)\}$, the vertex $u$ is guaranteed an odd color. 
Thus, at most $2(|K|-2)$ colorings of $v_1$ and $v_2$ must be avoided when proper $L$-coloring $v_1$ and $v_2$ to make $\vp'$  an odd coloring of $G$.
The number of ways to properly $L$-color $v_1$ and $v_2$ from $L(v_1)$ and $L(v_2)$ is at least 
\[|L(v_1)|(|L(v_2)|-1)\geq(m - 2\D(G) + |K| -1)(m - 2\D(G) + |K| -2) > 2(|K|-2).\]

Hence, we can color $v_1$ and $v_2$ to obtain an odd $m$-coloring of $G$, which is a contradiction.
\end{proof}

\begin{lemma}\label{lem:basic}
If $G$ is a  $K_{1,\ell+1}$-free graph, then $G$ has adjacent vertices $v_1, v_2$ such that 
$|N_G[v_1] \cap N_G[v_2]| \ge  \left\lceil \frac{\D(G)}{\ell}\right\rceil +1$.
Moreover, if $G$ is claw-free, then $G$ has a clique of size $\Delta(G)-|N_G[v_1]\cap N_G[v_2]|+2$.
\end{lemma}
\begin{proof}
Let $v_1$ be a vertex of $G$ with maximum degree.
Since $G$ is $K_{1,\ell+1}$-free, the graph induced by $N_G(v_1)$ has a maximum independent set $I$ of size at most $\ell$. 
Since every vertex in $N_G(v_1)\setminus I$ must have a neighbor in $I$, there is a vertex $v_2 \in I$ such that $|N_G(v_2) \cap N_G(v_1)| \ge \frac{\D(G)-|I|}{|I|}\geq\frac{\D(G)}{\ell}-1$.
Thus, $K = N_G[v_1] \cap N_G[v_2]$ satisfies $|K| \ge \left\lceil \frac{\D(G)}{\ell}\right\rceil +1$.
Moreover, if $G$ is claw-free, then $(N_G(v_1) \setminus N_G[v_2]) \cup \{v_1\}$ is a clique of size $\D(G)-|K|+2$.
\end{proof}

We now prove Theorem~\ref{thm:lcc}, Theorem~\ref{thm:genclaw}, and Theorem~\ref{thm:claw}.
 
\begin{thmm:lcc}
If $G$ is a graph with $\lcc(G)\le \ell$, then
$\chi_{o}(G)\le \frac{(2\ell-1)}{\ell}\D(G)+2$.
\end{thmm:lcc} 

\begin{proof}
For simplicity, let $\D(G)=\D$ and $m_G=\left\lfloor \frac{(2\ell-1)}{\ell}\D\right\rfloor+2=2\D-\left\lceil \frac{\D}{\ell}\right\rceil +2$.
Let $G$ be a minimal graph with $\lcc(G)\le \ell$ that is not odd $m_G$-colorable.
By considering a maximum clique containing a vertex of degree $\D$, we know there is a clique $K$ of size at least $\left\lceil \frac{\D}{\ell}\right\rceil+1 \ge 2$. 
By the minimality of $G$, $\odd(G-K)\le m_{G-K} \le m_G$.
Since $m_G\ge 2\D-|K|+3$,   Lemma~\ref{lem:bigclique} implies that $G$ is odd $m_G$-colorable, which is a contradiction.
\end{proof} 

\begin{thmm:genclaw}
For $\ell\ge 2$, if $G$ is a  $K_{1,\ell+1}$-free graph, then $\odd(G) \le \frac{(2\ell-1)}{\ell}\D(G)   + 1+\left\lceil \sqrt{\frac{2\D(G)}{\ell}+1}\right\rceil$.
\end{thmm:genclaw}
\begin{proof}
For simplicity, let $\D(G)=\D$ and  $m_G = \left\lfloor \frac{(2\ell-1)}{\ell}\D \right\rfloor + 1+\left\lceil \sqrt{\frac{2\D}{\ell}+1}\right\rceil$.
If $\D\le 2$, then $G$ is a path or a cycle, so the bound clearly holds.
If $3\le \D\le \ell$, then $\odd(G) \le \pcf(G)\le 2\D-1\le m_G$ by Theorem~\ref{thm:pcf:h-ver2}.
Now assume that $\D\ge \ell+1$. 
Let $G$ be a minimal $K_{1,\ell+1}$-free graph that is not odd $m_G$-colorable.
 By Lemma~\ref{lem:basic}, $G$ has adjacent vertices $v_1,v_2$ such that 
    $ |N_G[v_1] \cap N_G[v_2]| \ge  \left\lceil \frac{\D}{\ell}\right\rceil +1\geq3$.
    Let $K = N_G[v_1] \cap N_G[v_2]$.
By the minimality of $G$, $G-\{v_1, v_2\}$ has an odd $m_G$-coloring $\vp$ since $\odd(G-\{v_1, v_2\}) \le m_{G-\{v_1, v_2\}} \le m_G$.
Note that $\sqrt{\frac{2\Delta}{\ell}+1 } \ge\sqrt{ 2\left\lceil \frac{\Delta}{\ell}\right\rceil-1}= \sqrt{2(\left\lceil \frac{\Delta}{\ell}\right\rceil+1)-3}$,
 so 
$$m_G  \ge 
2\D-\left(\left\lceil \frac{\Delta}{\ell}\right\rceil+1\right)+ 
 2 +\sqrt{ 2\left(\left\lceil \frac{\Delta}{\ell}\right\rceil+1\right)-3}
 \ge  2\D-|K|+ 2+ \sqrt{2|K|-3},$$
where the last inequality holds since $f(x)=-x+\sqrt{2x-3}$ is a non-increasing function for $x\geq 2$. 
Thus, since $m_G$ is an integer, 
$m_G  \ge  2\D-|K|+ \left\lceil  2+ \sqrt{2|K|-3} \right\rceil
\ge  2\D-|K|+ \left\lceil  \frac{3}{2}+\sqrt{2|K|-3}\right\rceil$, so $G$ also has an odd $m_G$-coloring by Lemma~\ref{lem:bigclique2}, which is a contradiction.
\end{proof}

\begin{thmm:claw} 
If $G$ is a claw-free graph,
then $\odd(G)=\pcf(G) \le \frac{3}{2}\D(G)   + \left\lceil \sqrt{\D(G)} \right\rceil $.
\end{thmm:claw} 

\begin{proof}
For simplicity, let $\D(G)=\D$ and $m_G = \left\lfloor 1.5\D \right\rfloor +\left\lceil \sqrt{\D} \right\rceil$.
If $\D\le 2$, then $G$ is a path or cycle, so the bound clearly holds. 
If $\D =3$, then the bound holds by Theorem~\ref{thm:pcf:h-ver2}, so we may assume that $\D\ge 4$. 
Let $G$ be a minimal claw-free graph that is not odd $m_G$-colorable.
By Lemma~\ref{lem:basic},  $G$ has adjacent vertices $v_1,v_2$ such that 
$|N_G[v_1] \cap N_G[v_2]| \ge  \left\lceil \frac{\D}{2}\right\rceil +1$.
Let $K=N_G[v_1] \cap N_G[v_2]$.

Suppose that $|K| = \left\lceil \frac{\D}{2} \right \rceil +1$. 
By the moreover part of Lemma~\ref{lem:basic}, there is a clique $Q$ of size $\D -|K|+2 = \left\lfloor \frac{\D}{2} \right\rfloor +1$, which implies that $m_G\ge 2\D-|Q|+3$.
By the minimality of $G$,  $\odd(G-Q) \le m_{G-Q} \le m_G$, so Lemma~\ref{lem:bigclique} implies $G$  is odd $m_G$-colorable, which is a contradiction.

Suppose that $|K| \ge  \left\lceil \frac{\D}{2} \right\rceil +2$.
By the minimality of $G$, $\odd(G-\{v_1,v_2\}) \le m_{G-\{v_1,v_2\}} \le m_G$. Note that $$m_G = \left\lfloor 1.5\D \right\rfloor +\left\lceil \sqrt{\D} \right\rceil  \ge  2\D - |K|+  \left\lceil 3/2 + \sqrt{ 2|K|-3}\right\rceil.$$
Lemma~\ref{lem:bigclique2} implies $G$ is odd $m_G$-colorable, which is a contradiction.
\end{proof}

\section*{Acknowledgements}
Eun-Kyung Cho was supported by Basic Science Research Program through the National Research Foundation of Korea(NRF) funded by the Ministry of Education (No. RS-2023-00244543).
Ilkyoo Choi was supported by the Hankuk University of Foreign Studies Research Fund.
Hyemin Kwon and Boram Park were supported by the Basic Science Research Program through the National Research Foundation of Korea (NRF-2022R1F1A1069500).

\bibliographystyle{abbrvurl}
\bibliography{ref}

\end{document}